\documentclass{ijmart}

\usepackage{amsmath,amssymb,mathtools}
\usepackage[breaklinks=true,hidelinks]{hyperref}
\usepackage{microtype}
\usepackage{xcolor}

\makeatletter
\fancypagestyle{firstpage}{%
  \fancyhf{}%
  \cfoot{\thepage}}
\makeatother
\newtheorem{thm}{Theorem}[section]
\newtheorem{lem}[thm]{Lemma}
\newtheorem{prop}[thm]{Proposition}
\newtheorem{cor}[thm]{Corollary}
\theoremstyle{remark}

\newcommand{\conv}{\operatorname{conv}}
\newcommand{\dist}{\operatorname{dist}}
\newcommand{\diam}{\operatorname{diam}}
\newcommand{\vertx}{\operatorname{vert}}
\newcommand{\one}{\mathbf 1}

\begin{document}

\title{Facial distance and diameter}

\author[A. Deza]{Antoine Deza}
\address{McMaster University, Hamilton, Ontario, Canada}
\email{deza@mcmaster.ca}

\author[D. Martinez-Rubio]{David Mart{\'i}nez-Rubio}
\address{IMDEA Software Institute, Madrid, Spain}
\email{david.martinezrubio@imdea.org}

\author[J. F. Pena]{Javier F. Pe{\~n}a}
\address{Tepper School of Business, Carnegie Mellon University, Pittsburgh, Pennsylvania, USA}
\email{jfp@andrew.cmu.edu}

\author[E. Wirth]{Elias Wirth}
\address{The Voleon Group}
\email{wirth.elias.samuel@gmail.com}

\begin{abstract}
The facial distance of a polytope, equivalently its pyramidal width, is a
geometric condition measure that appears in the convergence bounds for 
Frank--Wolfe algorithms. Pe\~na and Wirth developed
calculus rules for this quantity and conjectured that the ratio of the facial distance 
to the diameter multiplied by the square root of the dimension is bounded by a constant. We show that this quantity is uniquely maximized by a regular simplex in each dimension, which implies that the conjecture holds and that $\sqrt{2}$ is the smallest constant. 
We  also compute the facial distance of the Birkhoff polytope.
\end{abstract}

\maketitle

\section{Introduction and main theorem}\label{sec:intro}

The facial distance is a geometric parameter measuring the separation between
the faces of a polytope and the convex hull of its remaining vertices.
Related distances between disjoint lattice polytopes arise in discrete geometry; 
see, for example, the study of kissing polytopes
\cite{DezaOnnPokuttaPournin2024,DezaLiuPournin2025,DezaLiuPournin2026}.

Lacoste-Julien and Jaggi \cite{LacosteJulienJaggi2015} introduced the
\emph{pyramidal width} in the analysis of Frank--Wolfe methods, and
Pe\~na and Rodr\'iguez \cite{PenaRodriguez2019} subsequently showed that it
coincides with the facial distance. 
The ratio of the facial distance to diameter provides a {\em condition measure} of a polytope that is central to the convergence analysis of various variants of the Frank--Wolfe method as featured in the articles~\cite{BomzeRinaldiZeffiro2024, LacosteJulienJaggi2015,IommazzoEtAl2026, PenaRodriguez2019, WirthPenaPokutta2026} and also discussed in the recent monograph of Braun et al. \cite{BraunEtAl2025}.

Pe\~na and Wirth \cite{PenaWirth2026} developed calculus rules for
the facial distance under products and linear transformations and in terms of
facet-exposing directions. Their results give exact values and bounds
for several families of polytopes.  The examples in~\cite{PenaWirth2026} led to the following conjecture: 

\begin{quote}
There exists
a universal constant $C$ such that every $d$-dimensional polytope $P$ satisfies
\begin{equation*}
 \frac{\Phi(P)}{\diam(P)}\le \frac{C}{\sqrt d}.
\end{equation*}
\end{quote}
Here $\Phi(P)$ and $\diam(P)$ denote the facial distance and diameter of $P$ respectively as recalled below.

Theorem~\ref{thm:main} proves the sharp dimension-dependent form of this conjecture. 
 Section~\ref{sec:proof} is devoted to its proof, which relies on a novel {\em face-nesting} construction.  
 Section~\ref{sec:birkhoff} considers the Birkhoff polytope and gives a new exact expression for its facial distance.

Throughout, $P$ denotes a $d$-dimensional polytope with
$d\ge2$, $\mathcal F$ denotes the set of its proper faces, that is, the faces of $P$ excluding $\emptyset$ and $P$,  and $\vertx(P)$ denotes the set of vertices of $P$.  We shall also assume that the ambient space of $P$ is endowed with the Euclidean norm.  Recall that the facial distance and diameter of $P$ are respectively defined as follows
$$
 \Phi(P)=\min_{F\in\mathcal F}
 \dist\bigl(F,\conv(\vertx(P)\setminus F)\bigr).
$$
and
$$
\diam(P) = \max_{x,y\in P} \|x-y\|.
$$

\begin{thm}\label{thm:main}
Let $P$ be a $d$-dimensional polytope. Then
\[
 \frac{\Phi(P)}{\diam(P)}
 \le
 \sqrt{\frac{d+1}{2\left\lfloor(d+1)^2/4\right\rfloor}},
\]
with equality if and only if $P$ is a regular simplex.
\end{thm}

\section{Proof of Theorem~\ref{thm:main}}\label{sec:proof}

The following {\em face nesting} technique provides a key component of the proof of Theorem~\ref{thm:main}.  For a face $F$ of $P$, we shall follow the convention that $\dim(F) = -1$ when $F=\emptyset.$

\begin{lem}\label{lem:pair}
Let $L\subseteq U$ be faces of $P$ with $\dim U-\dim L\ge2$. There are distinct
vertices $a,b$ and faces $f_a,f_b,h_a,h_b$  of $U$ such that
\[
 L\subsetneq f_a\subseteq h_b\subsetneq U,
 \qquad
 L\subsetneq f_b\subseteq h_a\subsetneq U,
\]
$a\in f_a\setminus h_a$, $b\in f_b\setminus h_b$, and 
\[
\dim(f_a) = \dim(f_b) = \dim(L)+1, \qquad \dim(h_a) = \dim(h_b) = \dim(U)-1.
\]
Consider the following two possible updates on $(L,U)$:
\[
\begin{array}{lll}
\text{keep $a$, discard $b$:}&(L,U)\leftarrow(f_a,h_b),\\[1mm]
\text{keep $b$, discard $a$:}&(L,U)\leftarrow(f_b,h_a).
\end{array}
\]
In either update $\dim L$ increases by one and $\dim U$ decreases by one.
When $L=\emptyset$, $a$ and $b$ may be any two distinct vertices of $U$.
\end{lem}

\begin{proof}
First assume $L\ne\emptyset$. Choose a face $f_a$ with
$L\subsetneq f_a\subsetneq U$, $\dim f_a=\dim L+1$, and a vertex
$a\in f_a\setminus L$. Since $a\notin L$ and $L$ is the intersection of the
facets of $U$ that contain it, some such facet $h_a$ does not contain $a$.
Because $\dim h_a\ge\dim L+1$,  we can choose a face $f_b$ of $h_a$ with
$L\subsetneq f_b$ and $\dim f_b=\dim L+1$, and a vertex $b\in f_b\setminus L$.
The faces $f_a$ and $f_b$ are distinct, and $L$ is a facet of each; hence
$f_a\cap f_b=L$. Thus $b\notin f_a$. Since $f_a$ is a proper face of $U$,
some facet $h_b$ of $U$ contains $f_a$ but not $b$.

If $L=\emptyset$ then the construction of $a,b$ and $f_a,f_b,h_a,h_b$ is simpler. Take any distinct vertices $a,b$ of $U$, put
$f_a=\{a\}$ and $f_b=\{b\}$. Since $a,b$ are distinct vertices of $U$, it follows that some facet $h_a$ of $U$ must contain $b$ but not
$a$, and some facet $h_b$ must contain $a$ but not $b$. 
\end{proof}

\begin{proof}[Proof of Theorem~\ref{thm:main}]
Let $\rho=\min\{\|u-v\|:u,v\in\vertx(P),\ u\ne v\}$.
Start with
$ L=\emptyset, U=P, s=0$,
and iterate the construction of Lemma~\ref{lem:pair}, choosing a closest pair  of vertices $a,b\in P$  in
the first iteration. Maintain the invariant that every kept vertex lies in $L$,
every discarded vertex lies outside $U$, and $s$ is the sum of the kept
vertices minus the sum of the discarded vertices. At each iteration the
parallelogram identity gives
\[
 \min\{\|s+(a-b)\|^2,\|s-(a-b)\|^2\}
 \le \|s\|^2+\|a-b\|^2.
\]
Update the pair $(L,U)$ as in
Lemma~\ref{lem:pair} 
choosing the update that gives the smaller norm on the updated $s$. If the kept and discarded vertices are denoted by
$p_i,q_i$, respectively, then after $t$ iterations
\[
 \|s\|^2 = \|\sum_{i=1}^t(p_i-q_i)\|^2 \le\sum_{i=1}^t\|p_i-q_i\|^2,
\]
where the first summand in the last expression is $\rho^2$ and every subsequent summand is at most
$\diam(P)^2$.\\

\noindent
{\sc Case 1: $d$ is odd}. Iterate the construction
$\lfloor(d+1)/2\rfloor$ times. Then the final lower and upper faces coincide.
If $p$ and $q$ are the averages of the kept and discarded vertices
respectively, then  $p\in L \in \mathcal F$, $q\in \conv(\vertx(P)\setminus L)$ and so
\[
 \Phi(P)^2\le\|p-q\|^2
 =\frac{\|s\|^2}{\lfloor(d+1)/2\rfloor^2}
 \le\frac{\rho^2+(\lfloor(d+1)/2\rfloor-1)\diam(P)^2}
 {\lfloor(d+1)/2\rfloor^2}
 \le\frac{\diam(P)^2}{\lfloor(d+1)/2\rfloor}.
\]

\noindent
{\sc Case 2: $d$ is even}. Iterate the construction
$\lfloor d/2\rfloor$ times. Let $p$ and $q$ be the averages of the kept and
discarded vertices, respectively. At termination,
$
 \dim L=\lfloor d/2\rfloor-1$ and $\dim U=\lfloor d/2\rfloor
$
with $p\in L$ and $q\in\conv(\vertx(P)\setminus U)$. Moreover,
\begin{equation}\label{eq:even-sharp-balancing}
 \|p-q\|^2
 \le \frac{\rho^2+(\lfloor d/2\rfloor-1)\diam(P)^2}
 {\lfloor d/2\rfloor^2}.
\end{equation}
Thus proceeding as in the case when $d$ is odd, we get the upper bound
\[
 \Phi(P)^2\le \frac{\diam(P)^2}
 {\lfloor d/2\rfloor}.
\]
With a bit more work, we can sharpen this upper bound as we next show.

Choose $c\in\vertx(U)\setminus L$. The $2\lfloor d/2\rfloor$ selected vertices
$p_1,\ldots,p_{\lfloor d/2\rfloor},q_1,\ldots,q_{\lfloor d/2\rfloor}$ are
distinct: at each iteration the new vertices lie in the current upper face and
outside the current lower face, whereas all previously kept vertices lie in
the lower face and all previously discarded vertices lie outside the upper
face. Their mean is $(p+q)/2$, and therefore
\begin{equation}\label{eq:even-sharp-variance}
\begin{aligned}
 \diam(P)^2
 &\ge \frac1{2\lfloor d/2\rfloor}
 \sum_{i=1}^{\lfloor d/2\rfloor}
 \bigl(\|p_i-c\|^2+\|q_i-c\|^2\bigr)\\
 &=\left\|\frac{p+q}{2}-c\right\|^2+
   \frac1{4\lfloor d/2\rfloor^2}
   \left(
   \sum_{1\le i<j\le \lfloor d/2\rfloor}
   \bigl(\|p_i-p_j\|^2+\|q_i-q_j\|^2\bigr)\right.\\
 &\hspace{43mm}\left.
   +\sum_{i,j=1}^{\lfloor d/2\rfloor}\|p_i-q_j\|^2
   \right)\\
 &\ge \left\|\frac{p+q}{2}-c\right\|^2
   +\frac{2\lfloor d/2\rfloor-1}{4\lfloor d/2\rfloor}\rho^2.
\end{aligned}
\end{equation}
There are $\lfloor d/2\rfloor(2\lfloor d/2\rfloor-1)$ terms in the two sums,
and each is at least $\rho^2$.
The vertices supporting $q$ lie outside $U$, hence outside $L$, and
$c\in U\setminus L$. Therefore
\[
 \frac{\lfloor d/2\rfloor q+c}{\lfloor d/2\rfloor+1}
 \in\conv(\vertx(P)\setminus L),
 \qquad
 \frac{\lfloor d/2\rfloor p+c}{\lfloor d/2\rfloor+1}\in U.
\]
Thus the faces $L$ and $U$ give two candidate distances.  Since $p\in L$ and $q\in \conv(\vertx(P)\setminus U)$,  averaging their
square distances gives
\begin{equation}\label{eq:even-sharp-candidates}
\begin{aligned}
 \Phi(P)^2 & \le \frac12\dist(L,\conv(\vertx(P)\setminus L))^2 + \frac12\dist(U,\conv(\vertx(P)\setminus U))^2
 \\
 &\le \frac12\left\|p-
 \frac{\lfloor d/2\rfloor q+c}{\lfloor d/2\rfloor+1}\right\|^2
 +\frac12\left\|q-
 \frac{\lfloor d/2\rfloor p+c}{\lfloor d/2\rfloor+1}\right\|^2\\
 &=\frac{(2\lfloor d/2\rfloor+1)^2}
 {4(\lfloor d/2\rfloor+1)^2}\|p-q\|^2
 +\frac1{(\lfloor d/2\rfloor+1)^2}
 \left\|\frac{p+q}{2}-c\right\|^2.
\end{aligned}
\end{equation}
Substituting \eqref{eq:even-sharp-balancing} and
\eqref{eq:even-sharp-variance} into \eqref{eq:even-sharp-candidates} yields
\[
\begin{aligned}
 \Phi(P)^2
 &\le \frac{(2\lfloor d/2\rfloor+1)^2}
 {4\lfloor d/2\rfloor^2(\lfloor d/2\rfloor+1)^2}
 \bigl(\rho^2+(\lfloor d/2\rfloor-1)\diam(P)^2\bigr)\\
 &\qquad +\frac1{(\lfloor d/2\rfloor+1)^2}
 \left(\diam(P)^2-
 \frac{2\lfloor d/2\rfloor-1}{4\lfloor d/2\rfloor}\rho^2\right).
 \end{aligned}
 \]
Therefore
\[
\begin{aligned}
 \Phi(P)^2
 &\le 
 \frac{2\lfloor d/2\rfloor+1}
 {2\lfloor d/2\rfloor(\lfloor d/2\rfloor+1)}\diam(P)^2\\
 &\qquad -
 \frac{2\lfloor d/2\rfloor^2+5\lfloor d/2\rfloor+1}
 {4\lfloor d/2\rfloor^2(\lfloor d/2\rfloor+1)^2}
 \bigl(\diam(P)^2-\rho^2\bigr)\\
 &\le \frac{2\lfloor d/2\rfloor+1}
 {2\lfloor d/2\rfloor(\lfloor d/2\rfloor+1)}\diam(P)^2
 =\frac{2(d+1)}{d(d+2)}\diam(P)^2.
\end{aligned}
\]
Hence the upper bound in Theorem~\ref{thm:main} holds regardless of the parity of $d$.

\medskip

We next show that equality in the bound holds if and only if $P$ is a regular simplex.
In both cases  ($d$ odd and $d$ even) equality in the bound forces $\rho=\diam(P)$. Hence equality in the upper bound implies that the
vertices of $P$ are pairwise equidistant, so  $P$ can only be a regular simplex.

Conversely, let $S$ be a regular $d$-simplex. For complementary faces with
$r$ and $d+1-r$ vertices, the segment joining their barycenters is orthogonal
to both faces and has squared length
\[
 \frac{d+1}{2r(d+1-r)}\diam(S)^2.
\]
This is minimized for $r=\lfloor(d+1)/2\rfloor$, giving
\[
 \frac{\Phi(S)}{\diam(S)}
 =\sqrt{\frac{d+1}{2\left\lfloor(d+1)^2/4\right\rfloor}}.
\]
\end{proof}

Pe\~na and Wirth \cite[Definition~2.1]{PenaWirth2026}
define the local facial distance $\Phi(F,P)$ for a nonempty face
$F$ of $P$ as follows
\[
\Phi(F,P) = \min_{G\in\mathcal F, G\subseteq F}
 \dist\bigl(G,\conv(\vertx(P)\setminus G)\bigr).
\] 
Corollary~\ref{cor:relative}
is the corresponding local version of Theorem~\ref{thm:main} applied to $F$.   Its proof is nearly identical to the proof of Theorem~\ref{thm:main} but starting with $U=F$ in lieu of $U=P$.

\begin{cor}\label{cor:relative}
Let $F$ be an $i$-dimensional face of $P$, with $i\ge1$. Then
\[
 \frac{\Phi(F,P)}{\diam(F)}
 \le
 \sqrt{\frac{i+1}{2\left\lfloor(i+1)^2/4\right\rfloor}}.
\]
\end{cor}


\section{The Birkhoff polytope}\label{sec:birkhoff}

Pe\~na and Wirth \cite[Example~6.4]{PenaWirth2026} show that
$\Phi(B_n)\ge1/\sqrt{n^2-1}$  for the Birkhoff polytope $B_n$ described below. The next result gives the exact value of $\Phi(B_n)$.

\begin{prop}\label{prop:birkhoff}
Let $n\ge3$ and 
$
 B_n=\{X\in\mathbb R^{n\times n}:X\ge0,\ X\one=\one,\ X^T\one=\one\}
$
with the Frobenius norm. Then
\[
 \Phi(B_n)=\frac1{\sqrt{\lfloor n^2/4\rfloor}}.
\]
\end{prop}

\begin{proof}
Consider a proper face $F$, and let
\[
 Z=\{(i,j):X_{ij}=0\text{ for every }X\in F\}.
\]
For $n\ge3$, within the affine space
\[
X\one=\one,\qquad X^T\one=\one,
\]
the facet-defining inequalities of $B_n$ are $X_{ij}\ge0$, with corresponding
facets $X_{ij}=0$. Since every face is the intersection of the facets containing it,
\[
F=B_n\cap\bigcap_{(i,j)\in Z}\{X_{ij}=0\}.
\]
Hence a permutation matrix $P$ lies outside $F$ if and only if there exists $(i,j)\in Z$ such that $P_{ij}=1$.
Let $A_{ij}=1$ on $Z$ and $0$ otherwise, and set $R=I-\one\one^T/n$.
If $X\in F$ and $Y\in\conv(\vertx(B_n)\setminus F)$, write $Y$ as a
convex combination of
permutation matrices outside $F$.  Each has an entry equal to $1$ for at least one $(i,j)\in Z$, while $X_{ij}=0$ for every $(i,j)\in Z$. Thus
\[
1\le\langle A,Y-X\rangle.
\]
Since $Y-X$ has zero row and column sums,
\[
\langle A,Y-X\rangle
=\langle RAR,Y-X\rangle
\le\|RAR\|_F\|Y-X\|_F.
\]
Moreover,
\[
\|RAR\|_F^2
\le\|AR\|_F^2
=\frac1n\sum_i
\left(\sum_jA_{ij}\right)
\left(n-\sum_jA_{ij}\right)
\le\left\lfloor\frac{n^2}{4}\right\rfloor,
\]
and therefore
\[
 \Phi(B_n)\ge\frac1{\sqrt{\lfloor n^2/4\rfloor}}.
\]

For the reverse inequality read indices modulo $n$. Set
\[
 Z=\{(i,i+s):1\le i\le n,\ 1\le s\le \lfloor n/2\rfloor\},
 \qquad
 F=B_n\cap\bigcap_{(i,j)\in Z}\{X_{ij}=0\}.
\]
For $e=(i,i+s)\in Z$, let $P_e$ be the transposition
$i\leftrightarrow i+s$, except when $n$ is even and $s=n/2$, where we use the
$3$-cycle
\[
 i\mapsto i+n/2,\qquad i+n/2\mapsto i+1,\qquad i+1\mapsto i.
\]
 Each $P_e$ is equal to $1$ at the coordinate indexed by $e$ and equal to $0$ at every coordinate indexed by $Z\setminus\{e\}$. Set
\[
 Y=\frac1{n\lfloor n/2\rfloor}\sum_{e\in Z}P_e.
\]
Thus $Y\in\conv(\vertx(B_n)\setminus F)$ and
$Y_{ij}=1/[n\lfloor n/2\rfloor]$ for $(i,j)\in Z$. Define
\[
X_{ij}=\begin{cases}
0,&(i,j)\in Z,\\[1mm]
Y_{ij}+1/[n\lceil n/2\rceil],&(i,j)\notin Z.
\end{cases}
\]
Each row and column contains $\lfloor n/2\rfloor$ entries of $Z$. Passing from
$Y$ to $X$ therefore removes total mass $1/n$ from each row and column and adds
back the same mass on the remaining $\lceil n/2\rceil$ entries. Hence $X$ is
doubly stochastic and $X_{ij}=0$ for every $(i,j)\in Z$, so $X\in F$. On $Z$ the difference
$Y-X$ equals $1/[n\lfloor n/2\rfloor]$, and off $Z$ it equals
$-1/[n\lceil n/2\rceil]$. Since $X\in F$ and
$Y\in\conv(\vertx(B_n)\setminus F)$,
\[
\begin{aligned}
\Phi(B_n)^2
&\le
n\lfloor n/2\rfloor\left(\frac1{n\lfloor n/2\rfloor}\right)^2
+n\lceil n/2\rceil\left(\frac1{n\lceil n/2\rceil}\right)^2\\
&=\frac1{\lfloor n/2\rfloor\lceil n/2\rceil}\\
&=\frac1{\lfloor n^2/4\rfloor}.
\end{aligned}
\]

\end{proof}


\begin{thebibliography}{10}

\bibitem{BomzeRinaldiZeffiro2024}
Immanuel~M. Bomze, Francesco Rinaldi and Damiano Zeffiro, \textsl{Frank--wolfe
  and friends: a journey into projection-free first-order optimization
  methods}, Annals of Operations Research \textbf{343} (2024), 607--638.

\bibitem{BraunEtAl2025}
G{\'a}bor Braun, Alejandro Carderera, Cyrille~W. Combettes, Hamed Hassani, Amin
  Karbasi, Aryan Mokhtari, and Sebastian Pokutta, \textsl{Conditional gradient
  methods: From core principles to {AI} applications}, MOS-SIAM Series on
  Optimization, vol.~35, Society for Industrial and Applied Mathematics, 2025.

\bibitem{DezaLiuPournin2025}
Antoine Deza, Zhongyuan Liu and Lionel Pournin, \textsl{Small kissing
  polytopes}, Vietnam Journal of Mathematics \textbf{53} (2025), no.~4,
  901--913.

\bibitem{DezaLiuPournin2026}
Antoine Deza, Zhongyuan Liu and Lionel Pournin, \textsl{Kissing polytopes in
  dimension 3}, Experimental Mathematics (2026).

\bibitem{DezaOnnPokuttaPournin2024}
Antoine Deza, Shmuel Onn, Sebastian Pokutta, and Lionel Pournin,
  \textsl{Kissing polytopes}, SIAM Journal on Discrete Mathematics \textbf{38}
  (2024), no.~4, 2643--2664.

\bibitem{IommazzoEtAl2026}
Gabriele Iommazzo, David Mart{\'i}nez-Rubio, Francisco Criado, Elias~Samuel
  Wirth, and Sebastian Pokutta, \textsl{Linear convergence of the
  {Frank}--{Wolfe} algorithm over product polytopes}, Proceedings of the 29th
  International Conference on Artificial Intelligence and Statistics,
  Proceedings of Machine Learning Research, vol. 300, 2026, pp.~1063--1071.

\bibitem{LacosteJulienJaggi2015}
Simon Lacoste-Julien and Martin Jaggi, \textsl{On the global linear convergence
  of {Frank}--{Wolfe} optimization variants}, Advances in Neural Information
  Processing Systems 28, 2015, pp.~496--504.

\bibitem{PenaRodriguez2019}
Javier~F. Pe{\~n}a and Daniel Rodr{\'\i}guez, \textsl{Polytope conditioning and
  linear convergence of the {Frank}--{Wolfe} algorithm}, Mathematics of
  Operations Research \textbf{44} (2019), no.~1, 1--18.

\bibitem{PenaWirth2026}
Javier~F. Pe{\~n}a and Elias Wirth, \textsl{Calculus of the facial distance},
  2026, arXiv:2608.28971v1 [math.OC].

\bibitem{WirthPenaPokutta2026}
Elias Wirth, Javier Pe{\~n}a and Sebastian Pokutta, \textsl{Fast convergence of
  {Frank}--{Wolfe} algorithms on polytopes}, Mathematics of Operations Research
  \textbf{51} (2026), no.~2, 1463--1485.

\end{thebibliography}
\bibliographystyle{ijmart}
\providecommand{\MR}{\relax\ifhmode\unskip\space\fi MR }
\providecommand{\MRhref}[2]{%
  \href{http://www.ams.org/mathscinet-getitem?mr=#1}{#2}
}
\providecommand{\href}[2]{#2}

\end{document}